\documentclass[draft]{amsart}
\usepackage{amssymb,latexsym,amscd, amsthm, epsfig,amsmath,amssymb}
\usepackage{color}
\usepackage{xypic}
\usepackage[all]{xy}

 \newtheorem{thm}{Theorem}[section]
 \newtheorem{cor}[thm]{Corollary}
 \newtheorem{lem}[thm]{Lemma}
 \newtheorem{prop}[thm]{Proposition}
 
 \theoremstyle{definition}
 \newtheorem{dfnt}[thm]{Definition}
 \theoremstyle{remark}
 \newtheorem{rem}[thm]{Remark}
 \theoremstyle{definition}
 \newtheorem{ex}[thm]{Example}

 \newcommand{\CC}{\mathbb{C}}
 \newcommand{\NN}{\mathbb{N}}

\def\move-in{\parshape=1.75true in 5true in}

\title[Monomials as sums of $k^{th}$-powers of  forms]{Monomials as sums of $k^{th}$-powers of forms}

\author[E. Carlini]{Enrico Carlini}
\address[E. Carlini]{DISMA- Department of Mathematical Sciences, Politecnico di Torino, Turin, Italy}
\email{enrico.carlini@polito.it}

\author[A. Oneto]{Alessandro Oneto}
\address[A. Oneto]{Department of Mathematics, Stockholm University, SE-106 91, Stockholm,  Sweden}
\email{oneto@math.su.se}

\begin{document}

\maketitle

\begin{abstract}
Motivated by recent results on the Waring problem for polynomial
rings \cite{FOS} and representation of monomial as sum of powers
of linear forms \cite{CCG}, we consider the problem of presenting
monomials of degree $kd$ as sums of $k^{th}$-powers of forms of
degree $d$. We produce a general bound on the number of summands
for any number of variables which we refine in the two variables
case. We completely solve the $k=3$ case for monomials in two and three variables.
\end{abstract}

\section{Introduction}
Let $S:=\bigoplus_{i\in\NN} S_i=\CC[x_0,\ldots,x_n]$ be the ring
of polynomials in $n+1$ variables with complex coefficients and
with the standard gradation. Given a homogeneous polynomial, or
form, $F\in S_k$ of degree $k\geq 2$, we can ask what is the
minimal number of linear forms needed to write $F$ as sum of their
$k^{th}$-power. The problems concerning this \textit{additive
decomposition of forms} are called \textit{Waring problems for
polynomials} and such minimal number is usually called
\textit{Waring rank}, or simply rank, of $F$.

In the last decades,  this kind of problems attracted a great deal
of work. In 1995, J.Alexander and A.Hirschowitz determined the
rank of the generic form \cite{AH}. However, given an explicit
form $F$, to compute the Waring rank of $F$ is more difficult and
we know the answer only in a few cases. One of these cases is the
monomial case.

In \cite{CCG}, E. Carlini, M.V. Catalisano and A.V. Geramita gave
an explicit formula to compute the Waring rank of a given monomial
in any number of variables and any degree.

In \cite{FOS}, R. Fr\"oberg, G. Ottaviani and B. Shapiro
considered a more general Waring problem. Given a form
$F\in\CC[x_0,\ldots,x_n]$ of degree $kd$, one can ask what is the
minimal number of forms of degree $d$ needed to write $F$ as sum
of their $k^{th}$-powers.

\begin{dfnt}
Let $F\in\CC[x_0,\ldots,x_n]$ be a form of degree $kd$ with $k\geq
2$, we set $$\#_k(F):=\min\{s~|~F=g_1^k+\ldots+g_s^k\}$$ where
$g_i$'s are forms of degree $d$. We call $\#_k(F)$ the
$k^{th}$\textit{-Waring rank of F}, or simply the $k^{th}$-rank of
$F$.
\end{dfnt}

Clearly, the $d=1$ case is the ``standard'' Waring problem. In
\cite{FOS}, the authors considered the $d\geq 2$ cases and they
proved that \textit{any} generic form of degree $kd$ can be
written as sum of $k^n$ $k^{th}$-powers.

Motivated by these recent results about the Waring rank of
monomials and about the mentioned generalization of the Waring
problem, we started to investigate the $k^{th}$-Waring rank for
monomials of degree $kd$.

In Section \ref{results}, we prove Theorem \ref{UpperBoundTHM}
stating that the $k^{th}$-rank of a monomial of degree $kd$ is
less or equal than $2^{k-1}$, for any $d$ and any number of
variables. Then we focus on some special cases: the binary case in
Section \ref{binarysection}, and the case with three or more
variables in Section \ref{threeormoresection}. In the binary
($n=1$) case we produce a general bound on $\#_k(M)$. While, for
$n\geq 2$, we give a complete description of the $k=3$ case.

The authors wish to thank R.Froeberg and B.Shapiro for their many
helpful suggestions and wonderful ideas. The first author was a
guest of the Department of Mathematics of the University of
Stockholm when this work was started. The first author received
financial support also from the KTH and the Fondo Giovani
Ricercatori of the Politecnico di Torino.

\section{Basic Facts}
First we recall the result about the Waring rank of monomials mentioned above.

\begin{thm}\cite{CCG}\label{CCGTheorem}
 Given a monomial $M=x_0^{a_0}\ldots x_n^{a_n}$ of degree $k$ such that $1\leq a_0\leq\ldots\leq a_n$, then
 \begin{equation}
  \#_k(M)={1 \over a_0+1}\prod_{i=0}^n (a_i+1).
 \end{equation}
\end{thm}

We now introduce some elementary tools to study the $k^{th}$-rank of monomials.
\begin{rem}\label{GroupingREM}
Consider a monomial $M$ of degree $kd$ in the variables
$\{x_0,\ldots,x_n\}$. We say that a monomial $M'$ of degree $kd'$
in the variables $\{X_0,\ldots,X_m\}$ is a \textit{grouping} of
$M$ if there exists a positive integer $l$ such that $d=ld'$ and $M$ can be obtained
from $M'$ by substituting each variable $X_i$ with a monomial of
degree $l$ in the $x$'s, i.e. $X_i=N_i(x_0,\ldots,x_n)$ for each
$i=1,\ldots,m$ with $\deg(N_i)=l$. The relation between the
$k^{th}$-rank of $M$ and $M'$ is given by $$\#_k(M')\geq\#_k(M).$$
Indeed, given a decomposition of $M'$ as sum of $k^{th}$-powers,
i.e. $$M'=\sum_{i=1}^r F_i(X_0,\ldots,X_n)^k,\text{  with
}\deg(F_i)=d',$$ we can write a decomposition for $M$ by using the
substitution given above, i.e. $$M=\sum_{i=1}^r
F_i(N_0(x_0,\ldots,x_n),\ldots,N_m(x_0,\ldots,x_n))^k.$$
\end{rem}

\begin{rem}\label{SpecializationREM}
Consider a monomial $M$ of degree $kd$ in the variables
$\{x_0,\ldots,x_n\}$. We say that a monomial $M'$ of the same
degree is a \textit{specialization} of $M$ if $M'$ can be found
from $M$ after a certain number of identifications of the type
$x_i=x_j$. Again, it makes sense to compare the two $k^{th}$-ranks
and we get $$\#_k(M)\geq\#_k(M').$$ Indeed, given a decomposition
of $M$ as sum of $r$ $k^{th}$-powers, we can write a decomposition
for $M'$ with the same number of summands applying the
identifications between variables to each addend.
\end{rem}

\begin{rem}\label{ProductREM}
Consider a monomial $M$ of degree $kd_1$ and $N$ a monomial of
degree $d_2$. We can look at the monomial $M'=MN^k$. Clearly the
degree of $M'$ is also divisible by $k$; again, it makes sense to
compare the $k^{th}$-rank of $M$ and $M'$. The relation is
$$\#_k(M)\geq\#_k(M').$$ Indeed, given a decomposition as sum of
$k^{th}$-powers for $M$, e.g. $M=\sum_{i=1}^r F_i^k$ with $F_i$'s
forms of degree $d_1$, we can easily find a decomposition for $M'$
with the same number of summands, i.e. $M'=MN^k=\sum_{i=1}^r
(F_iN)^k$.
\end{rem}

The inequality on the $k^{th}$-rank in Remark \ref{ProductREM} can be strict in
general as we can see in the following example.

\begin{ex}
Consider $k=3$ and the monomials $M=x_1x_2x_3$ and
$M'=(x_0^2)^3M=x_0^6x_1x_2x_3$. By Theorem \ref{CCGTheorem}, we
know that $\#_3(M)=\#_3(x_1x_2x_3)=4$, but we can consider a
\textit{grouping} of the monomial $M'$, i.e.
$M'=(x_0^3)^2(x_1x_2x_3)=X_0^2X_1$. By Remark \ref{GroupingREM} and
Theorem \ref{CCGTheorem}, we have $\#_3(M')\leq\#_3(X_0^2X_1)=3$.
\end{ex}

As a straightforward application of these remarks we get the
following lemma which is useful to reduce the number of cases to
consider once $k$ and $n$ are fixed.

\begin{lem}\label{ModuloLEMMA}
Given a monomial $M=x_0^{a_0}x_1^{a_1}\cdots x_n^{a_n}$ of degree
$kd$, then $$\#_k(M)\leq\#_k([M]),$$ where
$[M]:=x_0^{[a_0]_k}x_1^{[a_1]_k}\cdots x_n^{[a_n]_k}$, where the
$[a_i]_k$'s are the remainders of the $a_i$'s modulo $k$.
\end{lem}

\begin{proof}
We can write $a_i=k\alpha_i+[a_i]_k$ for each $i=0,\ldots,n$.
Hence, we get that $M=N^k[M]$, where
$N=x_0^{\alpha_0}x_1^{\alpha_1}\cdots x_n^{\alpha_n}$. Obviously,
$k|\deg([M])$ and by Remark \ref{ProductREM}, we are done.
\end{proof}

\begin{rem}\label{NumericalREM}
With the above notations and numerical assumptions, we have that
$[a_0]_k+\dots+[a_n]_k$ is a multiple of $k$ and also it has to be
at most $ (k-1)(n+1)=kn-n+k-1$. Hence, fixed the number of
variables $n+1$ and the integer $k$, we will have to consider only
a few cases with respect to the remainders of the exponents modulo
$k$.
\end{rem}

\section{Results on the $k^{\mathrm{th}}$-rank for monomials}

In this section we collect our results on the $k$-th rank of
monomials.

\subsection{The general case}\label{results}

Here we present some general results on the $k$-th Waring rank for
monomials.

\begin{rem}\label{k2REM}
Using the idea of grouping variables, we can easily get a complete
description of the $k=2$ case. Given a monomial $M$ of degree $2d$
which is not a square, we have $\#_2(M)=2$. Indeed,
$$M=XY=\left[\frac{1}{2}(X+Y)\right]^2+\left[\frac{i}{2}(X-Y)\right]^2,$$
where $X$ and $Y$ are two monomials of degree $d$.
\end{rem}

In the next result, we see that case $k=2$ is the unique in which
the $k^{th}$-rank of a monomial can be equal to two.

\begin{thm}\label{UpperBoundTHM}
If $M$ is a monomial of degree $kd$, then $\#_k(M)\leq 2^{k-1}$.
Moreover, $\#_k(M)=2$ if and only if $k=2$ and $M$ is not a
square.
\end{thm}
\begin{proof}
Any monomial $M\in S_{kd}$ is a specialization of the monomial
$x_1\cdot\ldots\cdot x_{kd}$. Now, we can consider the grouping
given by $$X_1=x_1\cdot\ldots\cdot
x_{d},\ldots,X_k=x_{(k-1)d+1}\cdots\ldots\cdot x_{kd}.$$ Thus, by
Remark \ref{SpecializationREM}, Remark \ref{GroupingREM} and
Theorem \ref{CCGTheorem}, we get the bound
$$\#_k(M)\leq\#_k(X_1\cdot\ldots\cdot X_k)=2^{k-1}.$$

Now suppose that $k>2$ and $\#_k(M)=2$. Hence, we can write
$M=A^k-B^k$ for suitable $A,B\in S_d$. Factoring we get
\[M=\prod_{i=1}^k (A-\xi_iB),\] where the $\xi_i$ are the $k^{th}$-roots of $1$. In
particular, the forms $A-\xi_iB$ are monomials. If $M$ is not a
$k^{th}$-power, using $A-\xi_1B$ and $A-\xi_2B$ we get that $A$
and $B$ are not trivial binomials. Hence a contradiction as
$A-\xi_3B$ cannot be a monomial. To conclude the proof we use the
$k=2$ case seen in Remark \ref{k2REM}.
\end{proof}

\begin{rem}\label{better upper bound} For $n\geq 2$ and $k$ small enough, we may observe that our result gives a better upper bound for the $k^{th}$-rank of monomials of degree $kd$ than the general result of \cite{FOS}. Indeed, if we look for which $k$ the
inequality $2^{k-1}\leq k^n$ holds, for $n=2$ we have $k\leq 6$ and, for $n=3$, $k\leq 9$. Increasing $n$,
we can find even better results, e.g. for $n=10$ our Theorem \ref{UpperBoundTHM} gives a better upperbound (for monomials)
for any $k\leq 59$.
\end{rem}

\subsection{Two variables case $(n=1)$.}\label{binarysection}

In the case of binary monomials, we can improve the upper bound
given in Theorem \ref{UpperBoundTHM}.

\begin{prop}\label{remainderprop}
 Let $M=x_0^{a_0}x_1^{a_1}$ be a binary monomial of degree $kd$. Then, $$\#_k(M)\leq\max\{[a_0]_k,[a_1]_k\}+1.$$
\end{prop}
\begin{proof}
By Lemma \ref{ModuloLEMMA}, we know that $\#_k(M)\leq\#_k([M])$;
hence, we consider the monomial $[M]=x_0^{[a_0]_k}x_1^{[a_1]_k}$. Now,
we observe that, as we said in Remark \ref{NumericalREM}, the
degree of $[M]$ is a multiple of $k$ and also $\leq 2k-2$; hence,
$\deg([M])$ is either equal to $0$, i.e. $[M]=1$, or $k$. In the
first case, it means that $M$ was a pure $k^{th}$-power, and the
$k^{th}$-rank is
$$\#_k(M)=1=\max\{[a_0]_k,[a_1]_k\}+1.$$
If $\deg([M])=k$, we can apply Theorem \ref{CCGTheorem} to $[M]$
and we get $$\#_k(M)\leq\#_k([M])=\max\{[a_0]_k,[a_1]_k\}+1.$$
\end{proof}

\begin{rem}\label{BinaryREM}
As a consequence of Proposition \ref{remainderprop}, for binary
monomials we have that $\#_k(M)\leq k$. Actually, this upper bound
can be directly derived from the main result in \cite{FOS}. We
observe that this upperbound is sharp by considering
$\#_k(x_0x_1^{k-1})=k$.
\end{rem}

As a consequence of Theorem \ref{UpperBoundTHM}, we are able to
easily give a solution for the $k=3$ case for binary monomials.

\begin{cor}\label{binarycorolk3}
 Given a binary monomial $M$ of degree $3d$, we have
 \begin{enumerate}
  \item $\#_3(M)=1$ if $M$ is a pure cube;
  \item $\#_3(M)=3$ otherwise.
 \end{enumerate}
\end{cor}
\begin{proof}
By Remark \ref{BinaryREM}, we have that the $3^{rd}$-rank can be
at most $3$; on the other hand, by Theorem \ref{UpperBoundTHM}, we
have that, $M$ is not a pure cube, the rank has to be at least
$3$.
\end{proof}

For $k\geq 4$ the situation is not so easily described and even in the case $k=4$
we have only partial results.

\begin{rem}
The first new step is to consider the $k=4$ case for binary
monomials. In such case we can only have rank $1,3$ or $4$.

Let $M=x_0^{a_0}x_1^{a_1}$ be a binary monomial of degree $4d$. By Remark
\ref{ModuloLEMMA}, we can consider the monomial $[M]$ obtained by
considering the exponents modulo $4$. Since $[M]$ has degree
divisible by $4$ and less or equal to $6$, we have to consider
only three cases with respect the remainders of the exponents
modulo $4$, i.e.
$$([a_0]_4,[a_1]_4)\in\{(0,0),(1,3),(2,2)\}.$$

The $(0,0)$ case corresponds to pure fourth powers, i.e. monomials
with $4^{th}$-rank equal to $1$. In the $(2,2)$ case we have
$$\#_4(M)\leq\#_4(x_0^2x_1^2)=3;$$ since the $4^{th}$-rank cannot
be two, we have that binary monomials in the $(2,2)$ class have
$4^{th}$-rank equal to three.

Unfortunately, we can not conclude in the same way the $(1,3)$
case. Since $\#_4(x_0x_1^3)=4$, a monomial in the $(1,3)$
class could still have rank equal to $4$. Indeed, for example, by using the computer
algebra system CoCoA, we have computed
$\#_4(x_0x_1^7)=\#_4(x_0^3x_1^5)=4$.
\end{rem}

A similar analysis can be performed for $k\geq 5$, but we can only
obtain partial results.

\subsection{$k=3$ case in three and more variables.}\label{threeormoresection}

In this section we consider the case $k=3$ with more than two
variables. By Theorem \ref{UpperBoundTHM}, we have that, also in
this case, we can only have $3^{rd}$-rank equal to $1,3$ or $4$.

This lack of space allows us to give a complete solution for monomials in three variables and degree $3d$.

\begin{prop}
 Given a monomial $M=x_0^{a_0}x_1^{a_1}x_2^{a_2}$ of degree $3d$, we have that
 \begin{enumerate}
  \item $\#_3(M)=1$ if $M$ is a pure cube;
  \item $\#_3(M)=4$ if $M=x_0x_1x_2$;
  \item $\#_3(M)=3$ otherwise.
 \end{enumerate}
\end{prop}
\begin{proof}
By Lemma \ref{ModuloLEMMA}, we consider the monomials $[M]$ with
degree divisible by $3$ and less or equal than $6$. Hence, we
have only four possible cases, i.e.
$$([a_0]_3,[a_1]_3,[a_2]_3)\in\{(0,0,0),(0,1,2),(1,1,1),(2,2,2)\}.$$

The $(0,0,0)$ case corresponds to pure cubes and then to monomials
with $3^{rd}$-rank equal to one. In the $(0,1,2)$ case we have, by
Theorem \ref{CCGTheorem}, $$\#_3(M)\leq\#_3(x_1x_2^2)=3;$$ since,
by Theorem \ref{UpperBoundTHM}, the rank of monomials which are
not pure cubes is at least $3$, we get the equality. Similarly, we
conclude that we have rank three also for monomials in the
$(2,2,2)$ class. Indeed, by using grouping and Theorem
\ref{CCGTheorem}, we have
$$\#_3(M)\leq\#_3(x_0^2x_1^2x_2^2)\leq\#_3(XY^2)=3.$$

Now, we just need to consider the $(1,1,1)$ class.

By Theorem \ref{CCGTheorem}, we have $\#_3(x_0x_1x_2)=4$. Hence,
we can consider monomials $M=x_0^{a_0}x_1^{a_1}x_2^{a_2}$ with
$a_0=3\alpha+1,~a_1=3\beta+1,~a_2=3\gamma+1$ and where at least one of
$\alpha,\beta,\gamma$ is at least one, say $\alpha>0$. By Remark
\ref{ProductREM}, we have
$$\#_3(M)=\#_3((x_0^{\alpha-1}x_1^{\beta}x_2^{\gamma})^3x_0^4x_1x_2)\leq\#_3(x_0^4x_1x_2).$$
Now, to conclude the proof, it is enough to show that
$\#_3(x_0^4x_1x_2)=3$. Indeed, we can write
$$x_0^4x_1x_2=\left[\sqrt{\frac{1}{6}}
x_0^2+x_1x_2\right]^3+\left[-\frac{1}{6}x_0^2+x_1x_2\right]^3+\left[\sqrt[3]{-2}x_1x_2\right]^3,$$

and thus we are done.
\end{proof}

Using the same ideas, we can produce partial results in the four
and five variables cases with $k=3$.

\begin{rem}
Given a monomial $M=x_0^{a_0}x_1^{a_1}x_2^{a_2}x_3^{a_3}$ with degree $3d$, we
consider the monomial $[M]$ which has degree divisible by $4$ and
less or equal than $8$. Hence, we need to consider only the
following classes with respect to the  remainders of the exponents
modulo $3$
$$([a_0]_3,[a_1]_3,[a_2]_3,[a_3]_3)\in\{(0,0,0,0),(0,0,1,2),(0,1,1,1),(0,2,2,2),(1,1,2,2)\}.$$
The $(0,0,0,0)$ case corresponds to pure cubes and we have rank
equal to one. Now, we use again Lemma \ref{ModuloLEMMA}, grouping
and Theorem \ref{CCGTheorem}.

In the $(0,0,1,2)$ class, we have $$\#_3(M)\leq\#_3(x_2x_3^2)=3;$$
in the $(0,2,2,2)$ class, we have $$\#_3(M)\leq\#_3(x_1^2x_2^2x_3^2)\leq\#_3(XY^2)=3;$$
in the $(1,1,2,2)$ class, we have $$\#_3(M)\leq\#_3(x_0x_1x_2^2x_3^2)\leq\#_3((x_0x_1)(x_2x_3)^2)=\#_3(XY^2)=3.$$

Again, since the $3^{rd}$-rank has to be at least three by Theorem \ref{UpperBoundTHM}, we conclude that in these classes the $3^{rd}$-rank is equal to three.

The $(0,1,1,1)$ class is a unique missing case because the upper
bound with $\#_3(x_1x_2x_3)=4$ is clearly useless. Another idea
would be to compute the $3^{rd}$-rank of $x_0^3x_1x_2x_3$. Indeed,
each monomial in four variables and degree $3d$ is of the type
$N^k(x_0^3x_1x_2x_3)$, hence, by Remark \ref{ProductREM}, we have
$$\#_3(M)\leq\#_3(x_0^3x_1x_2x_3).$$ Finding
$\#_3(x_0^3x_1x_2x_3)=3$, we would be done.
\end{rem}

\begin{rem}
Given a monomial $M=x_0^{a_0}x_1^{a_1}x_2^{a_2}x_3^{a_3}x_4^{a_4}$ with degree $3d$,
we consider the monomial $[M]$ which has degree divisible by $4$
and less or equal to $10$. Hence, we need to consider only the
following classes with respect to the remainders of the exponents
modulo $3$.

The $(0,0,0,0,0)$ class corresponds to pure cubes and $3^{rd}$-rank equal to one. By using Lemma \ref{ModuloLEMMA}, grouping, previous results in three or four variables and Theorem \ref{CCGTheorem}, we get the following results.

In the $(0,0,0,1,2)$ case, we have $$\#_3(M)\leq\#_3(x_3x_4^2)=3;$$
in the $(0,0,2,2,2)$ case, we have $$\#_3(M)\leq\#_3(x_1^2x_2^2x_3^2)=3;$$
in the $(0,1,1,2,2)$ case, we have $$\#_3(M)\leq\#_3(x_0x_1x_2^2x_3^2)=3;$$
in the $(1,2,2,2,2)$ case, we have $$\#_3(M)\leq\#_3(x_0x_1^2x_2^2x_3^2x_4^2)=\#_3((x_0x_1^2)(x_2x_3x_4)^2)\leq\#_3(XY^2)=3.$$

Hence, by Theorem \ref{UpperBoundTHM}, in these cases we have $3^{rd}$-rank equal to three.

There are only two missing cases: the $(0,0,1,1,1)$ case, which
can be reduced to the unique missing case in four variables seen
above; the $(1,1,1,1,2)$ case, for which it would be enough to
show that $\#_3(x_0x_1x_2x_3x_4^2)=3.$
\end{rem}

\section{Final remarks}

We conclude with some final remarks which suggest some projects
for the future.

\begin{rem} In this paper we work over the field of complex numbers. However,
for a monomial $M\in S_{kd}$ it is reasonable to look for a {\em
real} Waring decomposition, i.e. $M=\sum F_i^k$ where each $F_i$
has real coefficients. Even if Remarks
\ref{GroupingREM},\ref{SpecializationREM}, and \ref{ProductREM}
still hold over the reals, this is not longer true for Theorem
\ref{CCGTheorem}. This is the main obstacle to extend our results
over $\mathbb{R}$. However, in \cite{MR2811260} it is shown that
the degree $d$ monomial $x^ay^b$ is the sum of $a+b$, and no
fewer, $d$-th powers of real linear forms. Thus, we can easily
prove the analogue of Proposition \ref{remainderprop}. Let
$\#_k(M,\mathbb{R})$ be the {\em real} $k$-th rank, then
\[\#_k({x_0}^{a_0}{x_1}^{a_1},\mathbb{R})\leq [a_0]_d+[a_1]_d\]
and the bound is sharp. Notice that Corollary \ref{binarycorolk3}
cannot be extend to the real case as we cannot use Theorem
\ref{UpperBoundTHM}.
\end{rem}

\begin{rem} In \cite{CCG} it is proved that monomials in three variables
produce example of forms having (standard) Waring rank higher than
the generic form. This is not longer true, in general, for the
$k$-th rank. For example, in the $k=3$ case and in three
variables, the $3$-rd rank of a monomial is at most $4$. While,
for $d>>0$ the $3$-rd rank of the generic form of degree $3d$ is
$9$, see in \cite{FOS}.
\end{rem}


\begin{thebibliography}{CCG12}

\bibitem[AH95]{AH}
James Alexander and Andr{\'e} Hirschowitz.
\newblock Polynomial interpolation in several variables.
\newblock {\em Journal of Algebraic Geometry}, 4(2):201--222, 1995.

\bibitem[BCG11]{MR2811260}
Mats Boij, Enrico Carlini, and Anthony~V. Geramita.
\newblock Monomials as sums of powers: the real binary case.
\newblock {\em Proc. Amer. Math. Soc.}, 139(9):3039--3043, 2011.

\bibitem[CCG12]{CCG}
Enrico Carlini, Maria~Virginia Catalisano, and Anthony~V Geramita.
\newblock The solution to the waring problem for monomials and the sum of
  coprime monomials.
\newblock {\em Journal of Algebra}, 370:5--14, 2012.

\bibitem[FOS12]{FOS}
Ralf Fr{\"o}berg, Giorgio Ottaviani, and Boris Shapiro.
\newblock On the waring problem for polynomial rings.
\newblock {\em Proceedings of the National Academy of Sciences},
  109(15):5600--5602, 2012.

\end{thebibliography}
\end{document}